\documentclass[11pt,a4paper]{article}
\usepackage[utf8]{inputenc}
\usepackage{amsmath}
\usepackage{amsfonts}
\usepackage{amssymb}
\usepackage{amsthm}
\usepackage{graphicx}
\usepackage{enumitem}
\usepackage{lmodern}
\usepackage{microtype}
\usepackage{mathrsfs} 
\usepackage[english]{babel}
\usepackage[table]{xcolor}
\usepackage{tabularx} 
\usepackage{xspace}		
\usepackage{setspace}   
\usepackage[autostyle]{csquotes} 
\usepackage{xpatch} 
\usepackage{aliascnt} 
\usepackage{multirow}
\usepackage{array}
\usepackage{booktabs}
\usepackage{authblk}

\usepackage{url} 
\usepackage[hidelinks]{hyperref} 
\usepackage[capitalise,nameinlink,noabbrev]{cleveref}
\crefformat{equation}{#2(#1)#3}
\crefmultiformat{equation}{(#2#1#3)}{ and~(#2#1#3)}{, (#2#1#3)}{ and~(#2#1#3)}

\theoremstyle{definition}
\newtheorem{definition}{Definition}
\newaliascnt{example}{definition}

\aliascntresetthe{example}
\newaliascnt{question}{definition}

\aliascntresetthe{question}
\newaliascnt{conjecture}{definition}

\aliascntresetthe{conjecture}

\theoremstyle{plain}
\newtheorem{theorem}{Theorem}[section]
\newaliascnt{lemma}{theorem}
\newtheorem{lemma}[lemma]{Lemma}
\aliascntresetthe{lemma}
\newaliascnt{corollary}{theorem}
\newtheorem{corollary}[corollary]{Corollary}
\aliascntresetthe{corollary}
\newaliascnt{proposition}{theorem}
\newtheorem{proposition}[proposition]{Proposition}
\aliascntresetthe{proposition}

\theoremstyle{remark}

\addto\extrasenglish{%
}

\newcommand{\F}{\mathbb{F}}
\newcommand{\Fp}{\mathbb{F}_2}

\newcommand{\Fpn}{\mathbb{F}_{2^n}}
\newcommand{\Fpnv}{\mathbb{F}_2^n}
\newcommand{\Fpmv}{\mathbb{F}_2^m}

\newcommand{\Z}{\mathbb{Z}}

\newcommand{\GL}{\textnormal{GL}}

\newcommand{\Var}{\textnormal{Var}}
\newcommand{\degalg}{\deg_{\textnormal{alg}}}
\newcommand{\rank}{\textnormal{rank}}

\newcommand{\C}{\mathcal{C}}
\newcommand{\W}{\mathcal{W}}

\newcommand{\A}{\mathcal{A}}

\newcommand{\U}{\mathcal{U}}
\newcommand{\qbinom}[3][q]{{\genfrac[]{0pt}{}{#2}{#3}}_{#1}}
\newcommand{\set}[1]{\left\{#1\right\}}

\newcolumntype{C}[1]{>{\centering\let\newline\\\arraybackslash\hspace{0pt}}p{#1}}

\title{Nonvanishing $k$-flats of Boolean and vectorial functions}

\author{Christian Kaspers}
\affil{Otto von Guericke University Magdeburg, Faculty of Mathematics, Institute for Algebra and Geometry, 39106 Magdeburg, Germany\thanks{christian.kaspers@ovgu.de\\This is the revised extended abstract accepted for presentation at WCC 2026.}}
\date{}

\begin{document}

\maketitle

\begin{abstract}
	$k$th-order sum-free functions are a natural generalization of APN functions using the concept of (non)vanishing flats. In this paper, we introduce a new combinatorial technique to study the nonvanishing flats of Boolean functions. This approach allows us to determine the number of nonvanishing flats for an infinite family of Boolean functions. We moreover use it to show that any $k$th-order sum-free $(n,n)$-function of algebraic degree $k$ gives rise to an $(n-k)$th-order sum-free $(n,n)$-function of algebraic degree $n-k$. This implies the existence of millions of $(n-2)$th-order sum-free functions.
\end{abstract}

\medskip

\noindent \textbf{Key Words: } Boolean function, vectorial function, almost perfect nonlinear, $k$th-order sum-free, vanishing flats.

\medskip

\section{Introduction}
\label{sec:introduction}
	A $k$-dimensional affine subspace, in brief a $k$-flat, $A$ of $\Fpnv$ is said to be \emph{vanishing} with respect to an $(n,m)$-function $F$ if $\sum_{x \in A} F(x) = 0$. If this sum is nonzero, we call $A$ \emph{nonvanishing} with respect to $F$. Vanishing $k$-flats were introduced by \cite{li2020} for $k=2$. In this case, they are closely related to almost perfect nonlinear (APN) functions: An $(n,n)$-function is APN if and only if it has no vanishing $2$-flats \cite{hou2006,brinkmannleander2008}. APN functions have optimal differential properties. They have been intensively studied and, by now, several infinite families and millions of sporadic examples of these functions are known \cite{likaleyski2024,beierleetal2025preprint}.\par
	
	Because of the connection mentioned above, the number of vanishing $2$-flats is a good measure of how close a function is to an APN function. We refer to \cite{li2020} for more results on the number of vanishing $2$-flats of $(n,n)$-functions. For any $(n,m)$-function, this number is also directly related to the Walsh spectrum of the function~\cite{arshad2018}. Furthermore, vanishing $2$-flats can be used to construct $t$-designs from certain $(n,m)$-functions \cite{meidlpolujanpott2023}.\par 
	
	So while there are multiple results about vanishing $2$-flats, the situation is much less clear if $k \ge 3$. However, such vanishing $k$-flats recently came into focus with the introduction of $k$th-order sum-free functions: An $(n,n)$-function is called \emph{$k$th-order sum-free} if it has no vanishing $k$-flats \cite{carlet2025inverse}. By now, these functions have been studied in several papers, many focusing on the question for which values of $n$ and $k$ the inverse function $x^{2^n-2}$ is $k$th-order sum-free \cite{carlet2025inverse,ebeling2024,carlethou2024,houzhao2025a,houzhao2025b}. We refer to \cite{carlet2025generalizations} for the cryptographic background, an overview and more properties of $k$th-order sum-free functions. Clearly, APN functions are precisely second-order sum-free functions. For $k\ge 3$, however, $k$th-order sum-free functions seem to be rare and so far only one infinite family of these functions is known: the $(n,n)$-function $x^{1+2^j+\dots+2^{j(k-1)}}$ is $k$th-order sum-free if $\gcd(j,n)=1$ \cite{carlet2025generalizations,carlet2025tDegree}. Also regarding the number of nonvanishing $k$-flats of $(n,m)$-functions, so far no results for $k \ge 3$ are known.\par 
	
	In this paper, we study the nonvanishing $k$-flats of Boolean and vectorial functions. In \cref{sec:k=2}, we give an overview of the case $k=2$. In \cref{sec:technique}, we translate the problem of finding the nonvanishing $k$-flats of a Boolean function into finding matrices with prescribed sets of linearly independent columns. We believe that this new combinatorial technique, which does not use finite fields, is powerful and may lead to new theoretical and computational results. \par
	
	In \cref{sec:applications}, we apply our new approach to Boolean and vectorial functions. We show that each $k$th-order sum-free function of algebraic degree $k$ gives rise to an $(n-k)$th-order sum-free function of algebraic degree $n-k$ (\cref{th:complement_sum-free}). We also determine the number of nonvanishing $k$-flats of an infinite family of Boolean functions whose terms share a constant set of variables (\cref{th:d-intersecting_nonvanishingflats}). Because of page limitations we omit several computational results in this paper. We will include them in an extended version later. 
	
\section{Preliminaries}
\label{sec:preliminaries}	
	Let $k \in \{0, \dots, n\}$. We denote the set of all $k$-dimensional linear subspaces of $\Fpnv$, or \emph{$k$-subspaces} in short, by $\U_{n,k}$. The number of $k$-subspaces of $\Fpnv$ is
	\begin{equation}
		\label{eq:number_of_subspaces}
		|\U_{n,k}| = \qbinom[2]{n}{k} =\frac{(2^{n}-1)(2^{n-1}-1) \cdots (2^{n-k+1}-1)} {(2^{k}-1)(2^{k-1}-1)\cdots(2-1)},
	\end{equation}
	where $\qbinom[q]{\ }{}$ denotes the $q$-binomial coefficient.\par
	
	Let $U \in \U_{n,k}$. We write $U = \langle u_1,\dots,u_k\rangle$ if  $u_1,\dots,u_k \in \Fpnv$ span $U$. For $a \in \Fpnv$, we call $U+a := \{x+a : x \in U\}$ a $k$-dimensional \emph{affine subspace}, or in brief a \emph{$k$-flat}, of $\Fpnv$. Denote $\C_U := \set{U+a : a \in \Fpnv}$ and note that $|C_U| = 2^{n-k}$. We denote the set of all $k$-flats, of $\Fpnv$ by $\A_{n,k}$. Clearly,
	\begin{equation}
		\label{eq:number_of_flats}
		|\A_{n,k}| = 2^{n-k} |\U_{n,k}|.
	\end{equation}
	
	It is well known that $\A_{n,2}$ is the set of $4$-subsets of $\Fpnv$ whose elements sum to $0$.	Note that this does not hold if $k > 2$. For $k \ge 3$, the elements of a $k$-flat still sum to $0$, but not every $2^k$-set whose elements sum to~$0$ is a $k$-flat.\par
	
	In the following, we denote $[n] := \set{1,\dots,n}$. An \emph{$n$-variable Boolean function} is a function $f\colon \Fpnv\to\Fp$ which we usually consider by its algebraic normal form (ANF). If $I \subseteq [n]$, we write $x_I$ for the monomial $\prod_{i\in I}x_i$ and define $\Var(x_I) := I$. A \emph{vectorial function} is a function $F\colon \Fpnv \to \Fpmv$, in brief an \emph{$(n,m)$-function}. Note that an $(n,1)$-function is a Boolean function. We define an $(n,m)$-function usually by its ANF, i.\,e. the ANFs of its $n$-variable Boolean coordinate functions $f_1,\dots,f_m$. For $(n,n)$-functions, we sometimes also use their univariate representation on the finite field $\Fpn$. The \emph{algebraic degree} $\degalg(F)$ of an $(n,m)$-function $F$ is the maximum degree of its coordinate functions. If the ANF of an $(n,m)$-function contains only terms of degree $k$, we say that $F$ is \emph{$k$-homogeneous}.\par
		
	Denote by $N_k(F)$ the set of nonvanishing $k$-subspaces and by $N_{\A,k}(F)$ the set of nonvanishing $k$-flats of $F$, respectively, so we have $N_k(F) = \set{U\in \U_{n,k} : \sum_{x\in U}F(x) \ne 0}$ and $N_{\A,k}(F) = \set{A\in \A_{n,k} : \sum_{x\in A}F(x) \ne 0}$. In the following \cref{prop:algebraic_degree_subspace_sums,cor:degree_k_functions}, we show that if $F$ has algebraic degree $k$, then $N_{\A,k}(F)$ can be easily derived from $N_k(F)$.\par
	
	Let $F$ be an $(n,m)$-function of algebraic degree $k$. We define the \emph{derivative}~$D_aF$ of $F$ in direction $a\in\Fpnv$ as the $(n,m)$-function $D_aF$ with $D_aF(x) = F(x) + F(x+a) + F(a) + F(0)$, and the $\ell$th-order derivative $D_{a_1}\cdots D_{a_\ell}F$ of $F$ in direction $(a_1,\dots,a_\ell) \in(\Fpnv)^\ell$ by the composition $D_{a_1}\cdots D_{a_\ell}F (x) = (D_{a_\ell}F \circ \dots \circ D_{a_1}F)(x)$. It is easy to see that $D_{a_1}\cdots D_{a_\ell}F$ has algebraic degree at most $k-\ell$ and no constant term \cite[Proposition~5 and Corollary~1]{carlet2021book}. In particular, the $k$th-order derivative of $F$ is always $0$. 
	
	\begin{proposition}
		\label{prop:algebraic_degree_subspace_sums}
		Let $k \in [n]$, $U \in \U_{n,k}$, and let $F$ be an $(n,m)$-function. If the algebraic degree of $F$ is $k$, then $\sum_{x \in U} F(x) = \sum_{x \in A} F(x)$	for all $A \in \C_U$.
	\end{proposition}
	\begin{proof}
		Let $F$ be an $(n,n)$ function of algebraic degree $k$. For $c \ne a$, we have $D_aF(c) = \sum_{x \in \langle a,c\rangle} F(x)$ and, analogously, for $c \notin \langle a_1,\dots,a_\ell \rangle$, we have $D_{a_1}\cdots D_{a_\ell}F(c) = \sum_{x \in \langle c,a_1,\dots,a_\ell\rangle} F(x)$. Let $x_1,\dots, x_k$ form a basis of~$U$, and let $c \in \Fpnv \setminus U$. Then the $k$th-order derivative $D_{x_1}\cdots D_{x_k}F(c) = \sum_{x \in \langle c,x_1,\dots,x_k\rangle} F(x) = \sum_{x \in U} F(x) + \sum_{x \in U+c} F(x)$, which equals $0$ according to above. Hence, $\sum_{x \in U} F(x) = \sum_{x \in U+c} F(x)$.
	\end{proof}
	
	\cref{prop:algebraic_degree_subspace_sums} immediately implies \cref{cor:degree_k_functions}.
	\begin{corollary}
	\label{cor:degree_k_functions}
		Let $F$ be an $(n,m)$-function. If $F$ has algebraic degree $k$, then $|N_{\A,k}(F)| = 2^{n-k}|N_k(F)|$. In particular, in this case, $F$ is $k$th-order sum-free if and only if $|N_k(F)|=0$.
	\end{corollary}
	
	
	\cref{cor:sum-free-degree-bound} follows from the fact mentioned above that every $k$-th order derivative of a function of algebraic degree $k$ is $0$.

	\begin{corollary}[\cite{carlet2025generalizations}]
	\label{cor:sum-free-degree-bound}
		Let $F$ be an $(n,m)$-function, and let $k \in [n]$. If $\degalg(F) < k$, then $F$ has no nonvanishing $k$-flats, i.\,e. $N_{\A,k}(F) = \emptyset$. 
	\end{corollary}
	
	This implies that there are no $k$th-order sum-free functions of algebraic degree less than $k$, and considering \cref{cor:degree_k_functions}, it seems to be easier to find $k$th-order sum-free functions of algebraic degree $k$ than of any algebraic degree $r > k$. The $k$th-order sum-free functions from \cite{carlet2025generalizations,carlet2025tDegree} all have algebraic degree~$k$. So far, for $k > 2$, the only known example with $r > k$ is the inverse function~$x^{2^n-2}$ of algebraic degree~$n-1$: it is $(n-2)$th-order sum-free if $n$ is odd \cite{carlet2025inverse}. \cref{cor:sum-free-degree-bound} also leads to \cref{cor:add_smaller_degree}.
	
	\begin{corollary}
	\label{cor:add_smaller_degree}
		Let $k \in [n]$, and let $F$ and $G$ be $(n,m)$-functions, such that $\degalg(G) < k$. Then $N_{\A,k}(F+G) = N_{\A,k}(F)$. In particular, if $F$ is $k$th-order sum-free, so is $F+G$.
	\end{corollary}
	
	We continue by studying the behavior of $|N_{\A,k}(F)|$ under equivalence, see \cite{carlet2021book} for details on EA- and CCZ-equivalence. In general, $|N_{\A,k}(F)|$ is an EA-invariant but no CCZ-invariant \cite{carlet2025generalizations}. In the APN case, i.\,e. $k=2$, the number of nonvanishing $k$-flats is preserved under CCZ-equivalence \cite{li2020}. For $k\ge 3$, this does not hold in general: On $\F_{2^5}$, the permutation~$x^7$ is third-order sum-free. Its inverse~$x^9$, which is CCZ-equivalent to $x^7$, is quadratic and thus, by \cref{cor:sum-free-degree-bound}, not third-order sum-free \cite{carlet2025generalizations}.\par 
	
	\cref{cor:add_smaller_degree} motivates the following definition that introduces a generalization of EA-equivalence preserving the number of vanishing $k$-flats.
	
	\begin{definition}
		\label{def:degree-r-equivalent}
		Let $r \in \set{0,\dots,n}$. Two $(n,m)$-functions $F,G$ are said to be \emph{degree-$r$ equivalent} if there exists an affine $(n,n)$-permutation $L$, an affine $(m,m)$-permutation $M$, and an $(n,m)$-function $R$ with $\degalg(R)=r$ such that $G = M \circ F \circ L + R$.
	\end{definition}
	Note that degree-$1$ equivalence is precisely EA-equivalence. \cref{cor:sum-freedom_degree-r-invariant} is a direct consequence of \cref{cor:add_smaller_degree} and $|N_{\A,k}(F)|$ being EA-invariant.
	
	\begin{corollary}
		\label{cor:sum-freedom_degree-r-invariant}
		Let $F$ be an $(n,m)$-function. The number of nonvanishing $k$-flats $|N_{\A,k}(F)|$ is invariant under degree-$(k-1)$-equivalence. In particular, the $k$th-order sum-freedom of $F$ is a degree-$(k-1)$-invariant.
	\end{corollary}
	
	The \emph{Walsh transform} of an $n$-variable Boolean function~$f$ is the function $W_f \colon \Fpnv \to \Z$ defined by $W_f(a) = \sum_{x \in \F_2^n} (-1)^{f(x) + \langle x,a \rangle}$,	where $\langle\ ,\, \rangle$ denotes a scalar product on $\Fpnv$. The values of $W_f$ are called the \emph{Walsh coefficients} of~$F$, and the multiset $\W_f = \set{W_f(a) : a \in \Fpnv}$ is the \emph{Walsh spectrum} of~$f$. The Walsh spectrum of an $(n,m)$-function $F$ is the union $\W_f = \bigcup_{b\in\Fpnv, b\ne 0}  W_{f_b}$, where $f_b = \langle b,F(x_1,\dots,x_m)\rangle$, of the Walsh spectra of the component functions of $F$. The \emph{nonlinearity} $\text{nl}(F)$ of~$F$ is $2^{n-1} - \frac{1}{2}\max_{W \in \W_f}\left|W\right|$.\par
	
	An $n$-variable Boolean function with maximum nonlinearity~$2^{n-1}-2^{\frac{n}{2}-1}$ is a \emph{bent function}. Bent functions have Walsh coefficients $\pm 2^{\frac{n}{2}}$ and exist only if $n$ is even. An odd-dimension analogue are \emph{semi-bent functions} with Walsh coefficients $0, \pm 2^{\frac{n+1}{2}}$. However, for $n\ge9$ there exist functions with higher nonlinearity than semi-bent functions \cite{kavutetal2006,kavut2010}.
	
	In \cref{sec:technique}, we use a one-to-one correspondence between the $k$-subspaces of $\Fpnv$ and the $k \times n$-matrices over $\Fp$ of full rank~$k$ in reduced row echelon form. A matrix over $\Fp$ is said to be in \emph{reduced row echelon form (RREF)} if it is in echelon form and in every column with a leading entry all other entries are~$0$. Using elementary row operations, any matrix can be transformed into a unique RREF. So if we consider $U\in \U_{n,k}$ as the row span of a $k\times n$~matrix of rank $k$ and transform this matrix into RREF, we obtain a unique RREF matrix of rank $k$, whose rows form a basis of $U$. Vice versa, every full rank $k\times n$ matrix in RREF gives rise to a unique $k$-subspace of $\Fpnv$.\par
	From now on, for any $U \in \U_{n,k}$, we denote the corresponding $k\times n$ matrix of rank $k$ in RREF by $G_U$ and call it the \emph{RREF-generator matrix} of $U$. We will often consider matrices formed by some  columns of $G_U$: If $I \subseteq [n]$, we denote the matrix formed by the columns $i_1,\dots,i_{|I|} \in I$ of $G_U$ by $G_U[I]$.

\section{Nonvanishing $2$-flats and the Walsh spectrum}
\label{sec:k=2}
	According to \cite[Theorem~2.5]{arshad2018}, we can calculate the number of vanishing $2$-flats of an $(n,m)$-function $F$ from its Walsh spectrum: this number is $\frac{1}{24} \left( \frac{1}{2^{n+m}} \left(2^{4n} + \omega \right) - 3 \cdot 2^{2n} + 2^{n+1}\right)$ , where $\omega = \sum_{W \in \W_F} W^4$ is the sum of the fourth powers of the Walsh coefficients of $F$. We determine the number of nonvanishing $2$-flats $|N_{\A,2}|$ of $F$ by subtracting $|\A_{n,k}|$, see \cref{eq:number_of_flats}, from the expression above:
	
	\begin{theorem}
	\label{th:nonvanishing_2-flats_Walsh}
		Let $F$ be an $(n,m)$-function. The function $F$ has $|N_{\A,2}(F)| = \frac{2^{4n}(2^m-1)-\omega}{3\cdot2^{n+m+3}}$ nonvanishing $2$-flats. If $m=1$, then $|N_{\A,2}(f)| = \frac{2^{4n}-\omega}{3\cdot2^{n+4}}$.
	\end{theorem}
	
	The following result was also observed by \cite{arshad2018}. We add a short proof.
	\begin{theorem}
		\label{th:bent_best_nonvanishing}
		Let $f$ be an $n$-variable Boolean function. We have $|N_{\A,2}(f)| \le \frac{2^{2n-4}(2^n-1)}{3}$ and $\frac{N_{\A,2}(f)}{|\A_{n,2}|} \le \frac{2^{n-2}}{2^{n-1}-1}$. Equality holds if and only if $f$ is bent.
	\end{theorem}
	\begin{proof}
		To maximize $|N_{\A,2}(f)|$, according to \cref{th:nonvanishing_2-flats_Walsh}, we need to minimize $\omega$. Recall that the Walsh coefficients need to satisfy Parseval's relation $\sum_{W \in \W(f)}W^2 = 2^{2n}$. Thus, $\omega = \sum_{W \in \W(f)}W^4$ is minimal if $|W| = 2^{\frac{n}{2}}$ for all $W \in \W_f$ which is precisely the definition of a bent function and is only possible if $n$ is even. The second relation follows from simplifying the fraction, where $|N_{\A,2}(f)|$ is as in the first relation and $|\A_{2,k}|$ is as in \cref{eq:number_of_flats}.
	\end{proof}
	Thus, bent functions are the functions with the most nonvanishing $2$-flats and the situation is clear for even $n$. However, the question which functions have the most nonvanishing $2$-flats if $n$ is odd is still widely open.
	\begin{proposition}
		\label{prop:semi-bent}
		Let $n$ be odd, and let $f$ be an $n$-variable Boolean function. If $f$ is semi-bent, then $|N_{\A,2}(f)| = \frac{2^{2n-3}(2^{n-1}-1)}{3}$ and $\frac{N_{\A,2}(f)}{|\A_{n,2}|} = \frac{2^{n-1}}{2^n-1}$.
	\end{proposition}
	\begin{proof}
		The nonzero Walsh coefficients of $f$ are $\pm 2^{\frac{n+1}{2}}$ with a combined multiplicity of $2^{n-1}$. The result then follows from \cref{th:nonvanishing_2-flats_Walsh}.
	\end{proof}
	Contrary to $n$ even, semi-bent functions do not maximize the number of nonvanishing $2$-flats if $n$ is odd. This was known for $n = 9$, where \cite{arshad2018} showed that the functions from \cite{kavutetal2006} with higher nonlinearity (241) than semi-bent functions (240) also have more nonvanishing $2$-flats. We confirmed computationally that the functions from \cite{kavut2010} of nonlinearity~242 have even more nonvanishing $2$-flats. Furthermore, we verified that also functions with the same nonlinearity as semi-bent functions can have more nonvanishing $2$-flats: This is the case for the functions with five-valued Walsh spectrum from \cite[Theorems~7 and 9]{maitra2002} in dimension 5, 7 and 9.

\section{A new technique to determine the nonvanishing $k$-flats of Boolean functions}
\label{sec:technique}
	
	In this section, we present and apply our new technique to determine the nonvanishing $k$-flats of a Boolean function. We start with an easy lemma.
	
%
	
	\begin{lemma}
		\label{lem:odd_solutions}
		Let $v_1,\dots,v_k\in \Fp^k$. Then for any $v \in \Fp^k$, the equation $\alpha_1 v_1 + \dots + \alpha_k v_k = v$ with $\alpha_1,\dots, \alpha_k \in \Fp$ has an odd number of solutions if and only if $v_1,\dots,v_k$ are linearly independent.
	\end{lemma}
	\begin{proof}
		The solution set of $\alpha_1 v_1 + \dots + \alpha_k v_k = v$ is empty or a flat of $\Fp^k$. The only flats containing an odd number of elements are those of dimension~$0$ which implies that $v_1,\dots,v_k$ are linearly independent.
	\end{proof}
	
	In the proof of the following theorem we use the $n$-ary symmetric difference. Recall that for a family of sets $A = \set{A_1,\dots,A_n}$, the symmetric difference $\bigtriangleup A := A_1 \triangle \cdots \triangle A_n$ is defined as the set of elements occurring in an odd number of the sets $A_1,\dots,A_n$. The cardinality of $\bigtriangleup A$ is well known, we have
	\begin{equation}
		\label{eq:symmetric_difference}
		|\bigtriangleup A| = \sum_{\ell=1}^{n}(-2)^{\ell-1}\sum_{\set{i_1,\dots,i_\ell} \in \binom{n}{\ell}} \left|\bigcap_{j=1}^\ell A_{i_j}\right|,
	\end{equation}
	where $\binom{n}{\ell}$ denotes the set of all $\ell$-subsets of $[n]$.
	
	\begin{theorem}
	\label{th:symmetric_difference_subspaces}
		Let $f = m_1+\dots+m_t$ be a $k$-homogeneous $n$-variable Boolean function. We have $|N_k(f)| = \sum_{\ell=1}^{t} (-2)^{\ell-1} \sum_{\set{i_1,\dots,i_\ell} \in \binom{n}{\ell}} \left|\bigcap_{j=1}^\ell N_k(m_{i_j})\right|$.
	\end{theorem}
	\begin{proof}
		Let $U \in \U_{n,k}$. Then $U \in N_k(f)$ if and only if $U\in N_k(m_i)$ for an odd number of $i \in [t]$. Consequently, $N_k(f)$ is precisely the symmetric difference $N_k(f) = N_k(m_1) \triangle \cdots \triangle N_k(m_t)$, and the result follows from \cref{eq:symmetric_difference}.
	\end{proof}
	
	We next show how to determine $\bigcap_{i=1}^t N_k(m_i)$. We first present an approach for the nonvanishing $k$-flats of a single monomial $m$ and then extend it to the intersection of the sets of nonvanishing $k$-flats of $t$ monomials:
	\begin{theorem}
	\label{th:matrix_transformation}
		Let $m$ be an $n$-variable monomial of degree $k$. We have $N_k(m) = \set{U \in \U_{n,k} : \textnormal{rank}(G_U[\Var(m)]) = k}$ .
	\end{theorem}
	\begin{proof}
		Let $x = (x_1,\dots,x_n)$. Recall that $U \in N_k(m)$ if $\sum_{x\in U}m(x) \ne 0$ which is equivalent to $m(x)=1$ for an odd number of $x\in U$. Clearly, $m(x) = 1$ if and only if $x_i=1$ for all $i \in \Var(m)$. Thus, $U \in N_k(m)$ if and only if the number of $x \in U$ with $x_i = 1$ for all $i \in \Var(m)$ is odd. According to \cref{lem:odd_solutions} this is the case if and only if the rows of $G_U[\Var(m)]$ are linearly independent, since then $m(x)=1$ for a unique $x\in U$. Since $G_U[\Var(m)]$ is a quadratic $k\times k$ matrix, $U \in N_k(m)$ if and only if $\textnormal{rank}(G_U[\Var(m)]) = k$.
	\end{proof}
	
	\cref{th:matrix_transformation} extends naturally to the intersection of several monomials:
	\begin{corollary}
		\label{cor:matrix_transformation_many}
		Let $m_1,\dots, m_t$ be $n$-variable monomials of degree $k$. Then $\bigcap_{i=1}^t N_k(m_i) = \set{U \in \U_{n,k} : \textnormal{rank}(G_U[\Var(m_i)]) = k \text{ for }i \in [t]}$.
	\end{corollary}
	Combining \cref{th:symmetric_difference_subspaces} and \cref{cor:matrix_transformation_many}, we obtain a new technique to determine the number of nonvanishing $k$-flats of a Boolean function.
	
\section{Applications to Boolean and vectorial functions}
\label{sec:applications}
	In this section, we use the new approach presented in \cref{sec:technique} to determine the number of nonvanishing $k$-flats for an infinite family of Boolean functions and to study the nonvanishing $k$-flats of the so-called complement of an $(n,m)$-function. Note that we restrict ourselves mostly to $k$-homogeneous functions in this section because for them it is enough to study their vanishing $k$-subspaces according to \cref{cor:degree_k_functions}. For arbitrary functions of degree $k$, we may simply omit the terms of degree less than $k$, according to \cref{cor:add_smaller_degree}, and use the technique for the resulting $k$-homogeneous function.\par
	
	We first observe that the set of nonvanishing $k$-flats of the monomial $x_{[k]}$ with $x_{[k]}=x_1\cdots x_k$ has an easy description:
	\begin{proposition}
		\label{prop:123}
		For the $n$-variable monomial $x_{[k]}$, we have $N_k(x_{[k]}) = \set{U \in \U_{n,k} : G_U[[k]] = I_k}$ and $|N_k(x_{[k]})| = 2^{k(n-k)}$,	where $I_k$ denotes the $k\times k$ identity matrix.
	\end{proposition}
	\begin{proof}
		The first identity follows from \cref{th:matrix_transformation}. We then obtain $|N_k(x_{[k]})|$ by observing that the matrix $G_U$ has $2^{k(n-k)}$ entries that can be arbitrarily chosen from $\Fp$ to obtain a unique subspace $U\in N_k(x_{[k]})$.
	\end{proof}
	
	So for $x_{[k]}$ precisely those subspaces $U\in \U_{n,k}$ are nonvanishing for which $G_U$ is of the shape $\left(I_k|A_{n-k}\right)$, where $A_{n-k}$ is an arbitrary $k\times (n-k)$ matrix. Since we can obtain $x_{[k]}$ from any other monomial of degree $k$ by permuting the variables the following \cref{cor:number_monomials} holds.
	\begin{corollary}
		\label{cor:number_monomials}
		Let $m$ be an $n$-variable monomial of degree $k$. We have $|N_k(m)| = 2^{k(n-k)}$.
	\end{corollary}
	Analogously, by permuting the variables we can represent any Boolean function $f$ of degree $k$ in the form $f=x_{[k]} + g$ for some function $g$. This allows us to focus on matrices of shape $\left(I_k|R_{n-k}\right)$.	In \cref{th:d-intersecting_nonvanishingflats}, we precisely determine the number of nonvanishing $k$-flats for the following family of Boolean functions. 
	
	\begin{definition}
		\label{def:d-intersecting}
		Let $0 \le d \le n-1$, and let $f=m_1+\dots+m_t$ be an $n$-variable Boolean function. We say that $f$ is \emph{$d$-intersecting} if there exists a subset~$D \subseteq [n]$ with $|D| = d$ such that	$\Var(m_i) \cap \Var(m_j) = D$ for all $i,j \in [t]$ with $i \ne j$.
	\end{definition}
	Note that clearly $0 \le d \le \deg(f)-1$ and $\Var(m_1)\cap\dots\cap \Var(m_t) = D$. 
	
	\begin{lemma}
		\label{lem:d-intersecting_intersection}
		Let $f=m_1+\dots+m_t$ be a $d$-intersecting, $k$-homogeneous $n$\nobreakdash-variable Boolean function. Then 
		\begin{equation}
		\label{eq:d-intersecting_intersection}
			\left|\bigcap_{i=1}^t N_k(m_i)\right| = 2^{k(n-tk+(t-1)d)} \prod_{i=0}^{k-d-1}(2^k-2^{i+d})^{t-1}.
		\end{equation}
	\end{lemma}
	\begin{proof}
		Using \cref{cor:matrix_transformation_many}, we count all $U\in \U_{n,k}$ with $\rank(G_U[\Var(m_i)]) = k$ for all $i \in [t]$ . Without loss of generality suppose $m_1 = x_{[k]}$ and $D = [d]$ if $d > 0$. Then $G_U$ has shape $ (I_k | B_{{(k-d)}_1} | \cdots | B_{{(k-d)}_{t-1}} | A_{n-tk+(t-1)d})$, where $B_{{(k-d)}_i}$ is a $k\times(k-d)$ matrix such that $\rank(e_1,\dots,e_d | B_{{(k-d)}_i}) = k$ for all $i \in [t-1]$, and the entries of $A_{n-tk+(t-1)d}$ can be arbitrarily chosen.	It follows that we have $2^{k(n-tk+(t-1)d)}$ distinct choices for $A_{n-tk+(t-1)d}$ that we can combine with $\prod_{j=d}^{k-1}(2^k-2^{j})$ choices for $B_{{(k-d)}_i}$ for each $i \in [t-1]$. 
	\end{proof}
	
	\cref{lem:d-intersecting_intersection} allows us to prove our first main theorem.
		
	\begin{theorem}
	\label{th:d-intersecting_nonvanishingflats}
		Let $f$ be a $d$-intersecting, $k$-homogeneous $n$-variable Boolean function. We have
		\begin{equation}
		\label{eq:d-intersecting_nonvanishingflats}
			|N_k(f)| = \sum_{\ell=1}^{t} (-2)^{\ell-1} \binom{t}{\ell} 2^{k(n-\ell k+(\ell-1)d)} \prod_{i=0}^{k-d-1}(2^k-2^{i+d})^{\ell-1}.
		\end{equation}
	\end{theorem}
	\begin{proof}
		Let $f = m_1 + \cdots + m_t$. We use \cref{th:symmetric_difference_subspaces} to calculate $|N_k(f)|$ with $\left|\bigcap_{i=1}^t N_k(m_i)\right|$ as in \cref{eq:d-intersecting_intersection}, and take into account that for any $\ell \in [t]$ the intersection size $\left|\bigcap_{i=1}^t N_k(m_i)\right|$ is constant for all  $\set{i_1,\dots,i_{\ell}} \subseteq \binom{t}{\ell}$.
	\end{proof}
	
	In the case $d=0$, we can simplify \cref{th:d-intersecting_nonvanishingflats}:
	\begin{proposition}
	\label{prop:direct_sum_nonvanishingflats}
		Let $f$ be a $0$-intersecting, $k$-homogeneous $n$-variable Boolean function. And define $G = \prod_{i=0}^{k-1} (2^k-2^i)$. We have \[|N_k(f)| = \frac{2^{kn-1}}{G}\left(1-\left(1-\frac{G}{2^{k^2-1}}\right)^t\right).\]
	\end{proposition}
	\begin{proof}
		We consider \cref{eq:d-intersecting_nonvanishingflats} with $d=0$. Factoring out $\frac{2^{kn}}{-2G}$, we obtain	$|N_k(f)| = -\frac{2^{kn-1}}{G}\sum_{\ell=1}^{t}\binom{t}{\ell}\left(\frac{-G}{2^{k^2-1}}\right)^\ell.$	Adding and subtracting $1$ and using the binomial theorem, the right-hand side becomes $-\frac{2^{kn-1}}{G}\left(-1+\left(1-\frac{G}{2^{k^2-1}}\right)^t\right)$.
	\end{proof}
	Note that we can use \cref{prop:direct_sum_nonvanishingflats} to determine the number of nonvanishing $2$-flats of the bent function $x_1x_2+\cdots+ x_{n-1}x_n$ if $n$ is even and of the semi-bent function $x_1x_2+\cdots+ x_{n-2}x_{n-1}$ if $n$ is odd as these are $0$-intersecting functions, see \cref{th:bent_best_nonvanishing,prop:semi-bent}. \par 
	
	As another application of our new technique, we introduce the complement of a Boolean function and present a bijection between the nonvanishing $k$-flats of a function and those of its complement.
	
	\begin{definition}
	\label{def:complement_function}
		Let $m$ be an $n$-variable monomial. We call the $n$-variable monomial $m'$ with $\Var(m')=[n] \setminus \Var(m)$ the \emph{complement} of $m$ and denote it by $\overline{m}$. Analogously, for a Boolean function $f$ defined by $f=m_1+\dots+m_t$ we define its \emph{complement} by $\overline{m_1}+\dots+\overline{m_t}$ and denote it by $\overline{f}$.
	\end{definition}
	
	Clearly, the following three properties hold: $\overline{\overline{f}} = f$; if $f$ has degree $r$, then $\overline{f}$ has degree $n-r$; and if $f$ is homogeneous so is $\overline{f}$. Moreover, equivalence is preserved by taking the complement of two homogeneous functions:
	
	\begin{proposition}[\cite{hou1996,langevinleander2008}]
	\label{prop:equivalence_complements}
		Let $f,g$ be $r$-homogeneous $n$-variable Boolean functions. Then $f$ and $g$ are degree-$(r-1)$ equivalent if and only if $\overline{f}$ and $\overline{g}$ are degree-$(n-r-1)$ equivalent.
	\end{proposition}
	More precisely, if $g = f \circ L + h$ for some $L \in \GL(n,2)$ and a Boolean function~$h$ of algebraic degree at most $r-1$, then $\overline{g} = \overline{f} \circ (L^{-1})^{T} + h'$ for a Boolean function~$h'$ of algebraic degree at most $n-r-1$.\par
	
	To establish the aforementioned bijection, we recall the following definition: If $U$ is a $k$-subspace of an $n$-dimensional vector space~$V$, we call the set $U^\perp := \set{v \in V | \langle v,u\rangle = 0 \text{ for all } u \in U}$	the \emph{orthogonal complement} of $U$. It is well known that $U^\perp$ is an $(n-k)$-subspace of $V$. Note that $U^\perp = \ker G_U$.
	
	\begin{lemma}
	\label{lem:orthogonal_complement}
		Let $U \in \U_{n,k}$, and let $I$ be a $k$-subset of $[n]$. We have $\rank(G_U[I]) = k$ if and only if $\rank (G_{U^\perp}[[n]\setminus I]) = n-k$.
	\end{lemma}
	\begin{proof}
		We consider $U$ as a linear $[n,k]$ code with generator matrix $G_U$ whose columns $i_1,\dots,i_k \in I$ are linearly independent. Clearly, $U$ is equivalent to a code $U'$ with generator matrix $G_{U'} = G_U\sigma$, where $\sigma \in S_n, \sigma = (1\,i_i)\cdots (k\,i_k)$, permutes the columns of $G_U$ such that $i_1,\dots,i_k$ are the first columns of $G_{U'}$. We transform $G_{U'}$ into standard form $(I_k|P)$. Then the parity check matrix of $U'$ is $H_{U'} = (P^T|I_{n-k})$, and $H_U = H_{U'}\sigma$ is a parity check matrix of $U$. Clearly, the columns $j_1,\dots,j_{n-k} \in [n]\setminus I$ of $H_U$ are linearly independent, and we can transform~$H_U$ into $G_{U^\perp}$ by elementary row operations which preserve this property.
	\end{proof}
	
	\begin{theorem}
	\label{th:nonvanishing_flats_complement}
		Let $f$ be a $k$-homogeneous $n$-variable Boolean function, and let $U \in \U_{n,k}$. Then $U \in N_k(f)$ if and only if $U^\perp \in N_k(\overline{f})$. 
	\end{theorem}
	\begin{proof}
		The result follows from combining \cref{th:matrix_transformation} with \cref{lem:orthogonal_complement}.
	\end{proof}
	
	We remark that \cref{th:nonvanishing_flats_complement} does not hold if $f$ is not $k$-homogeneous: Suppose $m$ is a term of $f$ of degree $\ell < k$. According to \cref{cor:sum-free-degree-bound}, then $m$ sums to zero over every $k$-flat. However, in this case, $\overline{m}$ has degree $n-\ell > n-k$ and, thus, does not necessarily sum to $0$ over every $(n-k)$-flat. So in this case, $U \in N_k(f)$ does not imply $U^\perp \in N_k(\overline{f})$.	We obtain \cref{prop:best_n-2} as an immediate consequence of \cref{th:nonvanishing_flats_complement} and \cref{th:bent_best_nonvanishing}.
	\begin{theorem}
	\label{prop:best_n-2}
		Let $n$ be even, and let $f$ be an $n$-variable Boolean function of degree $n-2$. Then $|N_{\A,n-2}(f)| \le \frac{2^{2n-4}(2^n-1)}{3}$ and $\frac{N_{\A,n-2}(f)}{|\A_{n,2}|} \le \frac{2^{n-2}}{2^{n-1}-1}$. Equality holds if and only if $\overline{f'}$ is EA-equivalent to a bent function, where $f'$ is an $(n-2)$-homogeneous function that is degree-$(n-3)$ equivalent to $f$.
	\end{theorem}

	Eventually, we extend \cref{def:complement_function} and \cref{th:nonvanishing_flats_complement} to vectorial functions.
	\begin{definition}
	\label{def:complement_vectorial}
		Let $F$ be an $(n,m)$-function defined by Boolean coordinate functions $f_1,\dots,f_m$. We call the $(n,m)$-function $\overline{F}$ defined by the complements $\overline{f_1},\dots,\overline{f_m}$ of $f_1,\dots,f_m$ the \emph{complement} of $F$.
	\end{definition}
	
	Note that \cref{prop:equivalence_complements} transfers immediately to vectorial functions:
	\begin{corollary}
	\label{cor:equivalence_complements_vectorial}
		Let $F$ and $G$ be $r$-homogeneous $(n,m)$-functions. Then $F$ and $G$ are degree-$(r-1)$ equivalent if and only if $\overline{F}$ and $\overline{G}$ are degree-$(n-r-1)$ equivalent.
	\end{corollary}
	More precisely, if $G = M \circ F \circ L + H$ for $L\in \GL(n,2)$, $M\in \GL(m,2)$ and some $(n,m)$-function $H$ of algebraic degree at most $r-1$, then $\overline{G} = M \circ F \circ (L^{-1})^T + H'$ for some $(n,m)$-function $H'$ of algebraic degree at most $n-r-1$. The following result is a direct consequence of \cref{th:nonvanishing_flats_complement}:
	\begin{corollary}
	\label{cor:nonvanishing_flats_vectorial_complement}
		Let $F$ be a $k$-homogeneous $(n,m)$-function, and let $U \in \U_k$. Then $U \in N_k(F)$ if and only if $U^\perp \in N_k(\overline{F})$.
	\end{corollary}
	
	\cref{cor:nonvanishing_flats_vectorial_complement} in particular implies our second main theorem.
	\begin{theorem}
	\label{th:complement_sum-free}
		Let $F$ be a $k$-homogeneous $(n,m)$-function. Then $F$ is $k$-th order sum-free if and only if $\overline{F}$ is $(n-k)$th-order sum-free.
	\end{theorem}
	
	Thanks to \cref{cor:equivalence_complements_vectorial}, degree-$(k-1)$ inequivalent $k$-th order sum-free functions yield degree-$(n-k-1)$ inequivalent $(n-k)$th-order sum-free functions. Therefore, \cref{th:complement_sum-free} implies that every quadratic APN $(n,n)$-function $F$ gives rise to an $(n-2)$th-order sum-free $(n,n)$-function: We omit the affine terms in the ANF of $F$ to obtain $F'$, which preserves the APN property, and then $\overline{F'}$ is $(n-2)$th-order sum-free. Thus, the infinite families of quadratic APN functions, see \cite{likaleyski2024}, and the millions of sporadic quadratic APN functions on $\F_2^8$ by \cite{beierleetal2025preprint} lead to millions of inequivalent $(n-2)$th-order sum-free functions.
	
\section*{Acknowledgments}
	The author would like to thank Domingo Perez, Alexandr Polujan and Alexander Pott for the constructive discussions during this work. The author is funded by the Deutsche Forschungsgemeinschaft (DFG, German Research Foundation) – 541511634.
	

\bibliographystyle{alpha}
\bibliography{bib.bib} 

@phdthesis{arshad2018,
	title = {Contributions to the theory of almost perfect nonlinear functions},
	author = {Arshad, Razi},
	year = {2018},
	type={{PhD} thesis},
	address={Magdeburg},
	institution={Otto von Guericke University},
	url = {http://dx.doi.org/10.25673/13406},
	urn = {urn:nbn:de:gbv:ma9:1-1981185920-134691},
}

@misc{beierleetal2025preprint,
	title={Millions of inequivalent quadratic {APN} functions in eight variables}, 
	author={Christof Beierle and Philippe Langevin and Gregor Leander and Alexandr Polujan and Shahram Rasoolzadeh},
	year={2025},
	eprint={2508.04644},
	archivePrefix={arXiv},
	primaryClass={math.CO},
	url={https://arxiv.org/abs/2508.04644}, 
}

@article {brinkmannleander2008,
	AUTHOR = {Brinkmann, Marcus and Leander, Gregor},
	TITLE = {On the classification of {APN} functions up to dimension five},
	JOURNAL = {Des. Codes Cryptogr.},
	FJOURNAL = {Designs, Codes and Cryptography. An International Journal},
	VOLUME = {49},
	YEAR = {2008},
	NUMBER = {1-3},
	PAGES = {273--288},
	ISSN = {0925-1022},
	MRCLASS = {94A60 (11T71 94C10)},
	MRNUMBER = {2438456},
	DOI = {10.1007/s10623-008-9194-6},
	URL = {https://doi.org/10.1007/s10623-008-9194-6},
}

@book {carlet2021book,
	AUTHOR = {Carlet, Claude},
	TITLE = {Boolean functions for cryptography and coding theory},
	PUBLISHER = {Cambridge University Press, New York},
	YEAR = {2020},
	PAGES = {xiv+562},
	ISBN = {978-1-108-47380-4; [9781108606806]},
	MRCLASS = {94A60 (06E30 11T71 94D10)},
	MRNUMBER = {4625791},
	DOI = {10.1017/9781108606806},
	URL = {https://doi.org/10.1017/9781108606806},
}

@misc{carlethou2024,
	title={More on the sum-freedom of the multiplicative inverse function}, 
	author={Carlet, Claude and Hou, Xiang-dong},
	year={2024},
	eprint={2407.14660},
	archivePrefix={arXiv},
	primaryClass={math.NT},
	howpublished = {arXiv:2407.14660},
	url={https://arxiv.org/abs/2407.14660}, 
}

@article {carlet2025generalizations,
	AUTHOR = {Carlet, Claude},
	TITLE = {Two generalizations of almost perfect nonlinearity},
	JOURNAL = {J. Cryptology},
	FJOURNAL = {Journal of Cryptology. The Journal of the International
	Association for Cryptologic Research},
	VOLUME = {38},
	YEAR = {2025},
	NUMBER = {2},
	PAGES = {Paper No. 20, 32},
	ISSN = {0933-2790},
	MRCLASS = {94A60},
	MRNUMBER = {4870858},
	DOI = {10.1007/s00145-025-09538-5},
	URL = {https://doi.org/10.1007/s00145-025-09538-5},
}

@article {carlet2025inverse,
	AUTHOR = {Carlet, Claude},
	TITLE = {On the vector subspaces of {$\mathbb {F}_{2^n}$} over which the
	multiplicative inverse function sums to zero},
	JOURNAL = {Des. Codes Cryptogr.},
	FJOURNAL = {Designs, Codes and Cryptography. An International Journal},
	VOLUME = {93},
	YEAR = {2025},
	NUMBER = {4},
	PAGES = {1237--1254},
	ISSN = {0925-1022},
	MRCLASS = {11T06 (12E05 12E10)},
	MRNUMBER = {4894915},
	DOI = {10.1007/s10623-024-01531-6},
	URL = {https://doi.org/10.1007/s10623-024-01531-6},
}

@misc{carlet2025tDegree,
	author = {Claude Carlet},
	title = {A notion on {S}-boxes for a partial resistance to some integral attacks},
	howpublished = {Cryptology {ePrint} Archive, Paper 2024/1693},
	year = {2025},
	url = {https://eprint.iacr.org/2024/1693}
}

@misc{ebeling2024,
	title={On Sum-Free Functions}, 
	author={Alyssa Ebeling and Hou, Xiang-dong and Ashley Rydell and Shujun Zhao},
	year={2024},
	eprint={2410.10426},
	archivePrefix={arXiv},
	primaryClass={math.NT},
	howpublished={arXiv:2410.10426},
	url={https://arxiv.org/abs/2410.10426}, 
}

@article {hou1996,
	AUTHOR = {Hou, Xiang-dong},
	TITLE = {{${\textrm GL}(m,2)$} acting on {$R(r,m)/R(r-1,m)$}},
	JOURNAL = {Discrete Math.},
	FJOURNAL = {Discrete Mathematics},
	VOLUME = {149},
	YEAR = {1996},
	NUMBER = {1-3},
	PAGES = {99--122},
	ISSN = {0012-365X,1872-681X},
	MRCLASS = {94B05 (20B25)},
	MRNUMBER = {1375102},
	MRREVIEWER = {Jonathan\ I.\ Hall},
	DOI = {10.1016/0012-365X(94)00342-G},
	URL = {https://doi.org/10.1016/0012-365X(94)00342-G},
}

@article {hou2006,
	AUTHOR = {Hou, Xiang-dong},
	TITLE = {Affinity of permutations of {$\mathbb F_2^n$}},
	JOURNAL = {Discrete Appl. Math.},
	FJOURNAL = {Discrete Applied Mathematics. The Journal of Combinatorial
	Algorithms, Informatics and Computational Sciences},
	VOLUME = {154},
	YEAR = {2006},
	NUMBER = {2},
	PAGES = {313--325},
	ISSN = {0166-218X},
	MRCLASS = {06E30 (94B05)},
	MRNUMBER = {2194404},
	MRREVIEWER = {Marcel Wild},
	DOI = {10.1016/j.dam.2005.03.022},
	URL = {https://doi.org/10.1016/j.dam.2005.03.022},
}

@misc{houzhao2025a,
	title={Two Absolutely Irreducible Polynomials over $\mathbb{F}_2$ and Their Applications to a Conjecture by Carlet}, 
	author={Hou,Xiang-dong and Zhao, Shujun},
	year={2025},
	eprint={2502.04545},
	archivePrefix={arXiv},
	primaryClass={math.NT},
	howpublished={arXiv:2502.04545},
	url={https://arxiv.org/abs/2502.04545}
}

@misc{houzhao2025b,
	title={On a Conjecture About the Sum-Freedom of the Binary Multiplicative Inverse Function}, 
	author={Hou, Xiang-dong and Zhao, Shujun},
	year={2025},
	eprint={2504.21805},
	archivePrefix={arXiv},
	primaryClass={math.NT},
	howpublished={arXiv:2504.21805},
	url={https://arxiv.org/abs/2504.21805}
}

@incollection {kavutetal2006,
	AUTHOR = {Kavut, Sel\c{c}uk and Maitra, Subhamoy and Sarkar, Sumanta and
	Y\"{u}cel, Melek D.},
	TITLE = {Enumeration of 9-variable rotation symmetric {B}oolean
	functions having nonlinearity {$>240$}},
	BOOKTITLE = {Progress in cryptology---{INDOCRYPT} 2006},
	SERIES = {Lecture Notes in Comput. Sci.},
	VOLUME = {4329},
	PAGES = {266--279},
	PUBLISHER = {Springer, Berlin},
	YEAR = {2006},
	ISBN = {978-3-540-49767-7; 3-540-49767-6},
	MRCLASS = {94A60 (94B75 94C10)},
	MRNUMBER = {2454915},
	DOI = {10.1007/11941378\{_}19}

@article {kavut2010,
	AUTHOR = {Kavut, Sel\c{c}uk and Y\"{u}cel, Melek D.},
	TITLE = {9-variable {B}oolean functions with nonlinearity 242 in the
	generalized rotation symmetric class},
	JOURNAL = {Inform. and Comput.},
	FJOURNAL = {Information and Computation},
	VOLUME = {208},
	YEAR = {2010},
	NUMBER = {4},
	PAGES = {341--350},
	ISSN = {0890-5401,1090-2651},
	MRCLASS = {94C10},
	MRNUMBER = {2640837},
	DOI = {10.1016/j.ic.2009.12.002},
	URL = {https://doi.org/10.1016/j.ic.2009.12.002},
}

@incollection {langevinleander2008,
	AUTHOR = {Langevin, Philippe and Leander, Gregor},
	TITLE = {Classification of {B}oolean quartic forms in eight variables},
	BOOKTITLE = {Boolean functions in cryptology and information security},
	SERIES = {NATO Sci. Peace Secur. Ser. D Inf. Commun. Secur.},
	VOLUME = {18},
	PAGES = {139--147},
	PUBLISHER = {IOS, Amsterdam},
	YEAR = {2008},
	ISBN = {978-1-58603-878-6},
	MRCLASS = {94C10},
	MRNUMBER = {2581571},
}

@article {li2020,
	AUTHOR = {Li, Shuxing and Meidl, Wilfried and Polujan, Alexandr and
	Pott, Alexander and Riera, Constanza and St\u{a}nic\u{a}, Pantelimon},
	TITLE = {Vanishing flats: a combinatorial viewpoint on the planarity of
	functions and their application},
	JOURNAL = {IEEE Trans. Inform. Theory},
	FJOURNAL = {Institute of Electrical and Electronics Engineers.
	Transactions on Information Theory},
	VOLUME = {66},
	YEAR = {2020},
	NUMBER = {11},
	PAGES = {7101--7112},
	ISSN = {0018-9448},
	MRCLASS = {94D10},
	MRNUMBER = {4173630},
	MRREVIEWER = {Marco Calderini},
	DOI = {10.1109/TIT.2020.3002993},
	URL = {https://doi.org/10.1109/TIT.2020.3002993},
}

@article {likaleyski2024,
	AUTHOR = {Li, Kangquan and Kaleyski, Nikolay},
	TITLE = {Two new infinite families of {APN} functions in trivariate
	form},
	JOURNAL = {IEEE Trans. Inform. Theory},
	FJOURNAL = {Institute of Electrical and Electronics Engineers.
	Transactions on Information Theory},
	VOLUME = {70},
	YEAR = {2024},
	NUMBER = {2},
	PAGES = {1436--1452},
	ISSN = {0018-9448,1557-9654},
	MRCLASS = {94D10},
	MRNUMBER = {4703723},
	MRREVIEWER = {Jos\'{e}\ Luis\ Verdegay},
	DOI = {10.1109/tit.2023.3313148},
	URL = {https://doi.org/10.1109/tit.2023.3313148},
}

@article {maitra2002,
	AUTHOR = {Maitra, Subhamoy and Sarkar, Palash},
	TITLE = {Cryptographically significant {B}oolean functions with five
	valued {W}alsh spectra},
	JOURNAL = {Theoret. Comput. Sci.},
	FJOURNAL = {Theoretical Computer Science},
	VOLUME = {276},
	YEAR = {2002},
	NUMBER = {1-2},
	PAGES = {133--146},
	ISSN = {0304-3975,1879-2294},
	MRCLASS = {94A60 (94C10)},
	MRNUMBER = {1896350},
	DOI = {10.1016/S0304-3975(01)00196-7},
	URL = {https://doi.org/10.1016/S0304-3975(01)00196-7},
}

@article {meidlpolujanpott2023,
	AUTHOR = {Meidl, Wilfried and Polujan, Alexandr and Pott, Alexander},
	TITLE = {Linear codes and incidence structures of bent functions and
	their generalizations},
	JOURNAL = {Discrete Math.},
	FJOURNAL = {Discrete Mathematics},
	VOLUME = {346},
	YEAR = {2023},
	NUMBER = {1},
	PAGES = {Paper No. 113157, 22},
	ISSN = {0012-365X,1872-681X},
	MRCLASS = {94D10 (94B05)},
	MRNUMBER = {4481196},
	MRREVIEWER = {Pantelimon\ St\u{a}nic\u{a}},
	DOI = {10.1016/j.disc.2022.113157},
	URL = {https://doi.org/10.1016/j.disc.2022.113157},
}

\end{document}